\documentclass[11pt]{amsart}

\usepackage[T1]{fontenc}
\usepackage{lmodern}
\usepackage{amsmath,amssymb,amsthm,mathtools}
\usepackage{enumitem}
\usepackage{microtype}
\usepackage[colorlinks=true,linkcolor=blue,citecolor=blue,urlcolor=blue]{hyperref}
\usepackage[margin=1.15in]{geometry}
\allowdisplaybreaks

\newtheorem{theorem}{Theorem}[section]
\newtheorem{proposition}[theorem]{Proposition}
\newtheorem{lemma}[theorem]{Lemma}
\newtheorem{corollary}[theorem]{Corollary}
\theoremstyle{definition}
\newtheorem{definition}[theorem]{Definition}

\theoremstyle{remark}
\newtheorem{remark}[theorem]{Remark}

\newcommand{\Z}{\mathbb Z}

\newcommand{\C}{\mathbb C}
\newcommand{\Zp}{\mathbb Z_p}

\newcommand{\Gcal}{\mathcal G}
\newcommand{\Ical}{\mathcal I}

\newcommand{\Div}{\operatorname{Div}}
\newcommand{\Princ}{\operatorname{Pr}}
\newcommand{\Pic}{\operatorname{Pic}}
\newcommand{\Fitt}{\operatorname{Fitt}}
\newcommand{\Ann}{\operatorname{Ann}}
\newcommand{\charideal}{\operatorname{char}}
\newcommand{\val}{\operatorname{val}}

\newcommand{\coker}{\operatorname{coker}}
\newcommand{\im}{\operatorname{im}}
\newcommand{\kerop}{\operatorname{ker}}

\newcommand{\nr}{\mathrm{unr}}
\newcommand{\ram}{\mathrm{ram}}
\newcommand{\Norm}[1]{N_{#1}}
\newcommand{\aug}[1]{\Delta_{#1}}
\newcommand{\LambdaG}{\Lambda(\Gcal)}
\newcommand{\wtI}{\widetilde I}
\newcommand{\wtalpha}{\widetilde\alpha}
\newcommand{\wtIcal}{\widetilde{\mathcal I}}
\newcommand{\ThetaStar}{\Theta^{\ast}}

\title[Equivariant main conjecture for branched graph towers]{An Equivariant Main Conjecture for Branched $\mathbb Z_p$-Towers of Finite Graphs and Applications}

\author{Rusiru Gambheera}
\address{Morningside Center of Mathematics, Academy of Mathematics and Systems Science, Chinese Academy of Sciences, Beijing, China}
\email{gambheera@amss.ac.cn}

\subjclass[2020]{Primary 11R23, 05C25; Secondary 11M41, 05C31}
\keywords{Iwasawa theory of graphs, branched covers, Picard groups, Fitting ideals, Ihara $L$-functions, equivariant main conjecture}

\begin{document}

\begin{abstract}
Let $p$ be a prime and let $Y/X$ be a finite abelian branched cover of finite connected graphs. We consider a compatible branched $\mathbb Z_p$-tower above $Y$ whose total Galois group over $X$ is abelian. We construct an equivariant Picard Iwasawa module over the completed group ring of the total Galois group and give an explicit presentation of this module. On the analytic side, we construct an equivariant Artin--Ihara $L$-element at the top of the tower and isolate a modified factor whose value at $u=1$ is well defined. Our main result computes the Fitting ideal of the equivariant Picard Iwasawa module in terms of these special values and augmentation ideals determined by the inertia groups. As applications, we recover the main conjecture for branched $\mathbb Z_p$-towers from our earlier work and obtain finite-level Fitting-ideal and annihilation formulas which can be viewed as graph-theoretic analogues of the conjectures of Kurihara and Brumer--Stark.

\end{abstract}

\maketitle

\section{Introduction}

Let $p$ be a rational prime. For a finite connected graph $X$, the Picard
group of degree zero $\Pic^0(X)$ is a finite abelian group whose order is the
number of spanning trees of $X$. In the analogy between number theory and
graph theory, this group plays the role of the ideal class group of a number
field, while Ihara zeta and $L$-functions play the role of Dedekind zeta and
Artin $L$-functions. This analogy suggests studying the $p$-primary parts of
Picard groups in towers of finite connected graphs.

The Iwasawa theory of graphs was developed in the unramified setting by
Valli\`eres~\cite{Val21} and Gonet~\cite{Gonet22}, using Ihara $L$-functions
and Iwasawa modules, respectively. In~\cite{GV24}, we initiated the study of
branched $\Zp$-towers of finite connected graphs arising from voltage
assignments. For such a tower $(X_n)_{n\geq 0}$, we proved that there exist
$\mu,\lambda\in\Z_{\geq 0}$ and $\nu\in\Z$ such that
\[
\operatorname{ord}_p\bigl(|\Pic^0(X_n)|\bigr)
=\mu p^n+\lambda n+\nu
\]
for all sufficiently large $n$. Our proof was module theoretical.
In~\cite{GV26}, we developed the equivariant Ihara formalism needed to study
these towers analytically and related the characteristic ideal of the
Picard Iwasawa module to special values of Ihara $L$-functions. Fitting
ideals for finite and infinite unramified abelian covers were studied by
Kataoka in~\cite{Kataoka}.

The theory has also been developed in several directions.
DuBose and Valli\`eres studied $\mathbb Z_p^d$-towers~\cite{DV23},
Kleine and M\"uller proved a main conjecture for noncommutative
voltage covers~\cite{KleM26}, and Adachi, Mizuno, and Tateno
extended the theory to weighted graphs~\cite{AMT26}.
Kataoka established an analogue of Kida's formula for ramified
graph coverings~\cite{Kat25}.
Lei and Valli\`eres investigated spanning-tree numbers at primes
different from the tower prime~\cite{LV23}, while Lei and M\"uller
studied towers of isogeny graphs with level
structures~\cite{LM25} and formulated and proved main conjectures
for abelian towers of digraphs~\cite{LM26}.
Kundu and M\"uller studied ramified $\mathbb Z_p^d$-towers and
dual graphs~\cite{KuM26}, and developed a combinatorial approach
explaining how the location of ramified vertices affects
spanning-tree counts and Iwasawa invariants~\cite{KuM25}.

The goal of this paper is to prove an equivariant main conjecture for
branched $\Zp$-towers. Starting with a finite abelian branched cover $Y/X$,
we consider a compatible $\Zp$-tower above $Y$ and retain the action of the
entire group of the cover $Y_\infty/X$. We compute the Fitting ideal of the
resulting Picard Iwasawa module over the completed group ring of this
group. The formula involves special values of modified equivariant Ihara
$L$-polynomials and augmentation ideals determined by vertex inertia.

More precisely, let $Y/X$ be a finite abelian branched cover of finite
connected graphs with group $G$, and consider a tower
\[
Y=Y_0\longleftarrow Y_1\longleftarrow\cdots
\longleftarrow Y_n\longleftarrow\cdots\longleftarrow Y_\infty.
\]
We assume that the total group $\Gcal$ is abelian and fits into an exact
sequence
\[
1\longrightarrow\Gamma\longrightarrow\Gcal
\longrightarrow G\longrightarrow 1,
\qquad \Gamma\simeq\Zp.
\]
The tower is obtained from compatible voltage assignments and inertia
groups as in Section~3, and every finite graph $Y_n$ is assumed to be
connected. Writing $\Gamma_n=\Gamma^{p^n}$ and
$\Gcal_n=\Gcal/\Gamma_n$, we put
\[
\LambdaG=\Zp[[\Gcal]]
=\varprojlim_n\Zp[\Gcal_n].
\]
The equivariant Picard Iwasawa module is
\[
\Pic_{\LambdaG}(Y_\infty)
=\varprojlim_n\Pic_p(Y_n),
\qquad
\Pic_p(Y_n)=\Pic(Y_n)\otimes_{\Z}\Zp,
\]
where the transition maps are induced by the covering maps.

For each $v\in V_X$, let $\wtI_v\leq\Gcal$ be the stabilizer of a vertex of
$Y_\infty$ above $v$. We distinguish the vertices with nontrivial finite
inertia from those with infinite inertia by putting
\[
\begin{aligned}
B&=\{v\in V_X:1\ne\wtI_v\text{ is finite}\},\\
C&=\{v\in V_X:\wtI_v\text{ is infinite}\}.
\end{aligned}
\]
The remaining vertices are unramified. For a closed subgroup
$J\leq\Gcal$, we write
\[
\aug{J}=(j-1:j\in J)\subseteq\LambdaG
\]
for its augmentation ideal.

On the analytic side, we use the determinant formulas
of~\cite{GV26} to construct an equivariant Ihara $L$-element
\[
\Theta_{Y_\infty/X}(u)\in\LambdaG[[u]].
\]
One difficulty is that Ihara $L$-functions do not satisfy the inflation
property in general for branched covers. Thus compatibility of the
finite-level $L$-elements requires a separate argument. We obtain this
compatibility after passing to quotients by ideals generated by the norm
elements of the infinitely ramified vertices; these ideals have zero
inverse limit. If
$\ell=|V_X|$ and $m=|E_X|/2$, the resulting element has a factorization
\[
\Theta_{Y_\infty/X}(u)
=(1-u^2)^{m-\ell+|C|}
\ThetaStar_{Y_\infty/X}(u).
\]
Here $\ThetaStar_{Y_\infty/X}(u)\in\LambdaG[u]$ is a determinant
obtained by deleting the rows and columns corresponding to vertices in
$C$ from the norm-weighted Ihara matrix. In particular, its value at
$u=1$ is well defined.

For $B'\subseteq B$, let $Y_\infty^{B'}$ be the modified tower obtained
by replacing $\wtI_v$ by $\Gcal$ at the vertices $v\in B'$, while leaving
the other inertia groups and the voltage assignment unchanged. Thus these
vertices become totally ramified, and the set of infinitely ramified
vertices of $Y_\infty^{B'}/X$ is $B'\cup C$. We write
$\ThetaStar_{Y_\infty^{B'}/X}(u)$ for its modified equivariant
$L$-polynomial. Our main result is the following.

\begin{theorem}[Equivariant main conjecture; Theorem~\ref{thm:EMC}]
\label{thm:introduction-EMC}
With the notation above, one has the following equality of ideals in
$\LambdaG$:
\[
\Fitt_{\LambdaG}\!\left(\Pic_{\LambdaG}(Y_\infty)\right)
=
\left(
\ThetaStar_{Y_\infty^{B'}/X}(1)
\prod_{v\in B'\cup C}\aug{\wtI_v}
\ ;\ B'\subseteq B
\right).
\]
\end{theorem}

The ideal on the right is the sum, over $B'\subseteq B$, of the indicated
products of special values and augmentation ideals. Vertices with infinite
inertia contribute their augmentation ideals to every summand. For a
vertex with nontrivial finite inertia, its norm factor is retained inside
the determinant when $v\notin B'$, and its augmentation ideal occurs
outside the determinant when $v\in B'$. In this way, the formula records
both kinds of ramification over the integral completed group ring.

The proof starts with a description of the Picard Iwasawa module in terms
of divisors and principal divisors at the top of the tower. We use this
description to construct a free presentation whose relations come from
norm-weighted Laplacians and vertex inertia. A calculation of the maximal
minors of the presentation matrix, together with the determinant formulas
for the modified $L$-polynomials, then gives the theorem.

Two special cases recover earlier results. If $Y_\infty/X$ is
unramified, then $B=C=\varnothing$, and the theorem recovers the Picard-group
Fitting formula of Kataoka~\cite[Proposition~5.10]{Kataoka}.
If $G=1$, then $\Gcal=\Gamma\simeq\Zp$ and $B=\varnothing$. In this case,
each ramified inertia group has the form $\Gamma^{p^{k_v}}$, with
augmentation ideal $(\gamma^{p^{k_v}}-1)$ for a topological generator
$\gamma$ of $\Gamma$. The presentation is square, and the formula gives the
characteristic-ideal main conjecture of~\cite{GV24,GV26}.

The theorem also has the form of an integral equivariant main conjecture
in number theory: an Iwasawa module is related to equivariant $L$-values
together with local ideals arising from ramification. In particular, the
computation of the Fitting ideal of a modified classical Iwasawa module in
Gambheera--Popescu~\cite[Theorem~4.9]{GP25RW} and the refinement by the
author in~\cite[Theorem~1.2]{G25} provide arithmetic counterparts to this
description. Here the Picard Iwasawa module, modified Ihara values, and
vertex inertia supply the corresponding graph-theoretic objects.

The same presentation method also gives a finite-level result. Let
$Y/X$ be a finite abelian branched cover with group $G$, let $I_v\leq G$
be its vertex inertia groups, and put
$V_X^{\ram}=\{v\in V_X:I_v\ne1\}$. For
$B'\subseteq V_X^{\ram}$, let $Y^{B'}$ be the cover obtained by making the
vertices in $B'$ totally ramified. The modified polynomial
$\ThetaStar_{Y^{B'}/X}(u)$ is defined in Section~6 by the determinant with
the rows and columns indexed by $B'$ removed.
Theorem~\ref{thm:finite-Fitt} gives
\[
\Fitt_{\Zp[G]}\!\left(\Pic_p(Y)\right)
=
\left(
\ThetaStar_{Y^{B'}/X}(1)
\prod_{v\in B'}\aug{I_v}
\ ;\ B'\subseteq V_X^{\ram}
\right).
\]
This is a graph-theoretic analogue of Kurihara's formula for the Fitting
ideal of the minus part of the Pontryagin dual of a ray class group,
proved away from $2$ by
Dasgupta--Kakde~\cite[Theorem~3.5]{DK23}. In that formula, the
Sinnott--Kurihara ideal is constructed from equivariant $L$-values and
local factors at ramified primes. Here the generators are likewise
indexed by subsets of the ramified vertices, with modified Ihara values
and vertex augmentation ideals as their factors.

Since a Fitting ideal is contained in the annihilator of the module, we
also obtain
\[
\ThetaStar_{Y^{B'}/X}(1)\prod_{v\in B'}\aug{I_v}
\subseteq
\Ann_{\Zp[G]}\!\left(\Pic_p^0(Y)\right)
\]
for every $B'\subseteq V_X^{\ram}$.
In particular, $\ThetaStar_{Y/X}(1)$ annihilates $\Pic_p^0(Y)$.
These are graph-theoretic analogues of class-group annihilation
statements arising from the Brumer--Stark conjecture.

The paper is organized as follows. Section~2 recalls the graph-theoretic
and algebraic preliminaries. In Section~3, we introduce equivariant
branched towers and describe their Picard Iwasawa modules. Section~4
constructs the equivariant Ihara $L$-element and the modified
$L$-polynomials. In Section~5, we give a free presentation, prove the
equivariant main conjecture, and recover the one-variable main conjecture.
Section~6 contains the finite-level Fitting formula and its annihilation
consequences. Finally, Section 7 gives two explicit examples in which we compute the equivariant Fitting ideals and the Picard groups at finite levels.

\subsection*{Acknowledgments}

The author acknowledges the support of the AMS–Simons Travel Grant and the Morningside Center of Mathematics, Chinese Academy of Sciences. The author would also like to thank Daniel Vallières for several stimulating discussions.

\section{Preliminaries}

Throughout the paper, $p$ denotes a fixed rational prime. We use Serre's formalism for graphs, as in \cite{GV24,GV26}. Thus a graph $X$ consists of a vertex set $V_X$, a set $E_X$ of directed edges, an involution $e\mapsto \bar e$ without fixed points, and origin and terminus maps $o,t:E_X\to V_X$ satisfying $o(\bar e)=t(e)$. If $v\in V_X$, we write
\[
E_{X,v}=\{e\in E_X:o(e)=v\}.
\]
The graph $X$ is called locally finite if $E_{X,v}$ is a finite set for each $v\in V_X$. In that case we define the valency of a vertex $v\in V_X$ as $\val_X(v)=|E_{X,v}|$. An orientation of $X$ is a subset $O_X\subseteq E_X$ containing exactly one edge from each pair $\{e,\bar e\}$. If $X$ is finite, then
\[
\chi(X)=|V_X|-|O_X|=|V_X|-\frac{|E_X|}{2}.
\]

\subsection{Divisors and Picard groups}

For a graph $X$, let $\Div(X)$ be the free abelian group on $V_X$. If $X$ is locally finite, its Laplacian operator
\[
L_X:\Div(X)\longrightarrow \Div(X)
\]
is defined on vertices by
\begin{equation}\label{eq:laplacian}
L_X(v)=\val_X(v)v-\sum_{e\in E_{X,v}}t(e).
\end{equation}
We set
\[
\Princ(X)=\im(L_X),\qquad
\Pic(X)=\Div(X)/\Princ(X).
\]
Let $\Div^0(X)$ be the subgroup of divisors of degree zero and define
\[
\Pic^0(X)=\Div^0(X)/\Princ(X).
\]
When $X$ is finite and connected, $\Pic^0(X)$ is finite and its order is the number of spanning trees of $X$.

We use the $p$-adic versions
\[
\Div_p(X)=\Div(X)\otimes_\Z\Zp,\qquad
\Princ_p(X)=\Princ(X)\otimes_\Z\Zp,
\]
and similarly $\Pic_p(X)$ and $\Pic_p^0(X)$. If a finite abelian group $H$ acts on $X$ without inversion, then all these groups carry natural $\Zp[H]$-module structures.

\subsection{Branched covers from voltage assignments}

Let $H$ be an abelian group, written multiplicatively. A voltage assignment on a graph $X$ with values in $H$ is a function
\[
\alpha:E_X\longrightarrow H
\]
satisfying $\alpha(\bar e)=\alpha(e)^{-1}$. Let
\[
\Ical=\{(v,I_v):v\in V_X,\ I_v\le H\}
\]
be a family of subgroups indexed by the vertices of $X$. The graph
\[
X(H,\Ical,\alpha)
\]
has vertex set
\[
\coprod_{v\in V_X}\{v\}\times H/I_v
\]
and directed edge set $E_X\times H$. The edge $(e,\sigma)$ joins
\[
(o(e),\sigma I_{o(e)})\quad\text{to}\quad (t(e),\sigma\alpha(e)I_{t(e)}).
\]
The group $H$ acts on this graph by left multiplication on the second coordinate. The action is without inversion and free on directed edges. The natural map
\[
X(H,\Ical,\alpha)\longrightarrow X
\]
is a branched cover, and the stabilizer of a vertex above $v$ is $I_v$. We refer to \cite[Sections 3--4]{GV24} for further details.

For a finite subgroup $J$ of an abelian group $H$, we write
\[
\Norm{J}:=\sum_{j\in J}j
\in \Z[H] \]
for its norm element. If $J$ is a closed, topologically finitely generated subgroup of a profinite abelian group $H$, we write
\[
\aug{J}=(j-1:j\in J)\subseteq \Zp[[H]]
\]
for its augmentation ideal. If $g_1,\ldots,g_r$ topologically generate $J$, then
\begin{equation}\label{eq:augmentation-generators}
\aug{J}=(g_1-1,\ldots,g_r-1).
\end{equation}

\subsection{Fitting ideals}

Let $R$ be a commutative ring and let $M$ be a finitely presented $R$-module. Given a presentation
\[
R^a\xrightarrow{A}R^b\longrightarrow M\longrightarrow 0,
\]
the zeroth Fitting ideal $\Fitt_R(M)$ is generated by the $b\times b$ minors of $A$. We shall use repeatedly the standard inclusion
\begin{equation}\label{eq:fitt-ann}
\Fitt_R(M)\subseteq \Ann_R(M).
\end{equation}
We will also use the usual base-change property of Fitting ideals: if
$\varphi:R\to S$ is a morphism of commutative rings and $M$ is a finitely
presented $R$-module, then
\[
\operatorname{Fitt}_{S}\!\left(M\otimes_R S\right)
=
\operatorname{Fitt}_{R}(M)\,S,
\]
where $\operatorname{Fitt}_{R}(M)\,S$ denotes the ideal of $S$ generated by
the image of $\operatorname{Fitt}_{R}(M)$ under $\varphi$.

\section{Equivariant branched $\mathbb Z_p$-towers}

In this section, we introduce equivariant branched \(\mathbb Z_p\)-towers and describe their Picard Iwasawa modules in terms of divisors and principal divisors. This description will be used in Section 5 to construct an explicit presentation and compute the Fitting ideal.

\subsection{The basic setup}

Let $X$ be a finite connected graph and let
\[
Y=X(G,\Ical,\alpha)\longrightarrow X
\]
be a finite abelian branched cover, where $G$ is finite abelian,
\[
\Ical=\{(v,I_v):v\in V_X\},\qquad \alpha:E_X\longrightarrow G.
\]
Let $\Gamma$ be a multiplicative topological group isomorphic to $\Zp$, and suppose that $\Gcal$ is an abelian profinite group fitting into an exact sequence
\begin{equation}\label{eq:extension}
1\longrightarrow \Gamma\longrightarrow \Gcal\xrightarrow{\pi}G\longrightarrow 1.
\end{equation}
We assume that the voltage assignment $\alpha$ lifts to a voltage assignment
\[
\wtalpha:E_X\longrightarrow \Gcal,
\qquad \pi\circ\wtalpha=\alpha,
\]
and that, for every $v\in V_X$, we are given a closed subgroup $\wtI_v\le \Gcal$ satisfying
\[
\pi(\wtI_v)=I_v.
\]
We put
\[
\wtIcal=\{(v,\wtI_v):v\in V_X\}
\]
and define
\[
Y_\infty=X(\Gcal,\wtIcal,\wtalpha).
\]
Thus $Y_\infty/X$ is a branched cover with group $\Gcal$, and $Y/X$ is the quotient corresponding to the closed subgroup $\Gamma$.

Fix a topological generator $\gamma$ of $\Gamma$. For $n\ge 0$, set
\[
\Gamma_n=\Gamma^{p^n},\qquad \Gcal_n=\Gcal/\Gamma_n,
\]
and let $\pi_n:\Gcal\twoheadrightarrow \Gcal_n$ be the natural projection. We put
\[
I_{v,n}=\pi_n(\wtI_v),\qquad
\Ical_n=\{(v,I_{v,n}):v\in V_X\},\qquad
\alpha_n=\pi_n\circ\wtalpha,
\]
and define
\[
Y_n=X(\Gcal_n,\Ical_n,\alpha_n).
\]
Then $Y_0\simeq Y$ and we obtain a tower of branched covers
\[
Y=Y_0\longleftarrow Y_1\longleftarrow Y_2\longleftarrow\cdots\longleftarrow Y_n\longleftarrow\cdots\longleftarrow Y_\infty.
\]
We assume throughout that all finite graphs $Y_n$ are connected.

Let
\[
\LambdaG=\Zp[[\Gcal]]=\varprojlim_n\Zp[\Gcal_n].
\]
We shall distinguish three types of vertices:
\begin{align*}
A&=\{v\in V_X:\wtI_v=1\},\\
B&=\{v\in V_X:1\ne \wtI_v\text{ and }\wtI_v\text{ is finite}\},\\
C&=\{v\in V_X:\wtI_v\text{ is infinite}\}.
\end{align*}
Thus $V_X=A\sqcup B\sqcup C$.

For $v\in C$, the closed subgroup $\wtI_v\cap\Gamma$ is nontrivial and hence open in $\Gamma$. Consequently, there is a unique integer $k_v\ge0$ such that
\begin{equation}\label{eq:kv}
\wtI_v\cap\Gamma=\Gamma^{p^{k_v}}.
\end{equation}
Choose once and for all a set $\wtI_v^{\,0}\subseteq\wtI_v$ mapping bijectively to $I_v$ under $\pi$. Thus
\[
|\wtI_v^{\,0}|=|I_v|,
\qquad
\wtI_v=\wtI_v^{\,0}\Gamma^{p^{k_v}}
\]
as a disjoint union of cosets of $\Gamma^{p^{k_v}}$.

We use the standard elements
\[
\omega_n=\gamma^{p^n}-1\in\LambdaG,
\qquad
\omega_{n,k}=\frac{\gamma^{p^n}-1}{\gamma^{p^k}-1}
=\sum_{j=0}^{p^{n-k}-1}\gamma^{jp^k}
\]
for $n\ge k$.

\subsection{Divisors at the top of the tower}

For each $v\in V_X$, write
\[
w_{v,\infty}=(v,\wtI_v)\in V_{Y_\infty}
\]
for the distinguished vertex above $v$. We define the equivariant divisor module at the top of the tower by
\begin{equation}\label{eq:div-top}
\Div_{\LambdaG}(Y_\infty)
=
\bigoplus_{v\in V_X}\Zp[[\Gcal/\wtI_v]]\,w_{v,\infty}.
\end{equation}
The summands are viewed as $\LambdaG$-modules through the natural quotient maps. This description agrees canonically with the inverse limit of the modules $\Div_p(Y_n)$ under the push-forward maps induced by the covering maps $Y_{n+1}\to Y_n$.

For each $n$, let
\[
w_{v,n}=(v,I_{v,n})\in V_{Y_n}.
\]
The natural projection $Y_\infty\to Y_n$ induces a surjective $\LambdaG$-module morphism
\begin{equation}\label{eq:pi-n-div}
\varpi_n:\Div_{\LambdaG}(Y_\infty)\longrightarrow\Div_p(Y_n),
\qquad w_{v,\infty}\longmapsto w_{v,n}.
\end{equation}

\begin{lemma}\label{lem:kernel-varpi}
Assume that $n\ge k_v$ for every $v\in C$. Then
\[
\kerop(\varpi_n)=\omega_n\Div_{\LambdaG}(Y_\infty).
\]
\end{lemma}

\begin{proof}
It is enough to work one vertex orbit at a time. If $v\in A\cup B$, then $\wtI_v\cap\Gamma=1$. Hence for every $m\ge n$, the natural projection gives an exact sequence
\[
1\longrightarrow \Gamma_n/\Gamma_m
\longrightarrow \Gcal_m/I_{v,m}
\longrightarrow \Gcal_n/I_{v,n}
\longrightarrow 1.
\]
Passing to completed group rings gives an exact sequence
\[
0\longrightarrow \omega_n\Zp[[\Gcal/\wtI_v]]
\longrightarrow \Zp[[\Gcal/\wtI_v]]
\longrightarrow \Zp[\Gcal_n/I_{v,n}]
\longrightarrow 0.
\]

Now let $v\in C$. If $n\ge k_v$, then $\Gamma_n\subseteq\wtI_v$ by \eqref{eq:kv}, and therefore
\[
\Gcal/\wtI_v\xrightarrow{\sim}\Gcal_n/I_{v,n}.
\]
Thus the map on the $v$-component is an isomorphism. At the same time $\omega_n$ acts trivially on $\Zp[[\Gcal/\wtI_v]]$. Combining the components gives the result.
\end{proof}

\subsection{Principal divisors}

Consider the unramified voltage cover
\[
Y_\infty^{\nr}=X(\Gcal,\{(v,1):v\in V_X\},\wtalpha).
\]
Its divisor module
\[
\Div_{\LambdaG}(Y_\infty^{\nr})
=
\bigoplus_{v\in V_X}\LambdaG\,w_{v,\infty}^{\nr}
\]
is free of rank $|V_X|$, where $w_{v,\infty}^{\nr}=(v,1)$. The natural morphism
\[
\iota:Y_\infty^{\nr}\longrightarrow Y_\infty
\]
induces a surjection on divisor modules. For $v\in V_X$, define
\begin{equation}\label{eq:P-v-infty}
P_{v,\infty}
:=\iota_*L_{Y_\infty^{\nr}}(w_{v,\infty}^{\nr})
=\val_X(v)w_{v,\infty}
-\sum_{e\in E_{X,v}}
\bigl(t(e),\wtalpha(e)\wtI_{t(e)}\bigr).
\end{equation}
For later use, we also introduce the corresponding divisors at the finite
levels. For each $n\geq 0$, let
\[
Y_n^{\mathrm{unr}}
=
X\bigl(\mathcal G_n,
\{(v,\{1\}) : v\in V_X\},
\alpha_n\bigr)
\]
be the unramified cover obtained by replacing all the inertia groups by the
trivial group, and let
\[
\iota_n:Y_n^{\mathrm{unr}}\longrightarrow Y_n
\]
be the natural morphism of graphs. If $w_{v,n}^{\mathrm{unr}}$ denotes the
distinguished vertex of $Y_n^{\mathrm{unr}}$ lying above $v$, we define
\[
P_{v,n}
:=
(\iota_n)_*
\bigl(
L_{Y_n^{\mathrm{unr}}}(w_{v,n}^{\mathrm{unr}})
\bigr)
\in \operatorname{Div}_p(Y_n).
\]
Explicitly,
\[
P_{v,n}
=
\operatorname{val}_X(v)\,w_{v,n}
-
\sum_{e\in E_{X,v}}
\bigl(t(e),\alpha_n(e)I_{t(e),n}\bigr),
\]
where $w_{v,n}=(v,I_{v,n})$ denotes the distinguished vertex of $Y_n$
lying above $v$. In particular, if
\[
\varpi_n:
\operatorname{Div}_{\Lambda(\mathcal G)}(Y_\infty)
\longrightarrow
\operatorname{Div}_p(Y_n)
\]
denotes the natural projection map, then
\[
\varpi_n(P_{v,\infty})=P_{v,n}.
\]
We introduce the $\LambdaG$-submodules of $\Div_{\LambdaG}(Y_\infty)$
\begin{align}
\Princ^A&=\big\langle P_{v,\infty}:v\in A\big\rangle,\label{eq:PrA}\\
\Princ^B&=\big\langle \Norm{\wtI_v}P_{v,\infty}:v\in B\big\rangle. \label{eq:PrB}
\end{align}
For $v\in C$ and $n\ge k_v$, set
\[
\mathcal N_{v,n}:=\Norm{\wtI_v^{\,0}}\,\omega_{n,k_v}\in\LambdaG,
\]
where \[
\Norm{\wtI_v^{\,0}}=\sum_{i\in \wtI_v^{\,0}} i \in \LambdaG
\]
and define
\begin{equation}\label{eq:PrC}
\Princ^{C,n}
=\big\langle \mathcal N_{v,n}P_{v,\infty}:v\in C\big\rangle.
\end{equation}
Finally, let
\begin{equation}\label{eq:Nn}
\mathcal N_n
=\omega_n\Div_{\LambdaG}(Y_\infty)+\Princ^A+\Princ^B+\Princ^{C,n}.
\end{equation}

\begin{proposition}\label{prop:finite-codescent}
Assume that $n\ge k_v$ for all $v\in C$. Then $\varpi_n$ induces a natural isomorphism
\[
\Div_{\LambdaG}(Y_\infty)/\mathcal N_n
\xrightarrow{\sim}
\Pic_p(Y_n).
\]
\end{proposition}

\begin{proof}
By Lemma \ref{lem:kernel-varpi}, it is enough to identify the image of the remaining terms in \eqref{eq:Nn}. For $v\in A$, the image of $P_{v,\infty}$ is $P_{v,n}$. For $v\in B$, the map $\wtI_v\to I_{v,n}$ is an isomorphism and hence
\[
\varpi_n(\Norm{\wtI_v}P_{v,\infty})
=\Norm{I_{v,n}}P_{v,n}.
\]
If $v\in C$, then the coset decomposition of $\wtI_v$ gives
\[
\pi_n(\mathcal N_{v,n})=\Norm{I_{v,n}},
\]
so that
\[
\varpi_n(\mathcal N_{v,n}P_{v,\infty})
=\Norm{I_{v,n}}P_{v,n}.
\]
By the principal-divisor description for branched voltage covers in \cite[Corollary 4.7]{GV24}, the elements $\Norm{I_{v,n}}P_{v,n}$, as $v$ ranges through $V_X$, generate $\Princ_p(Y_n)$ over $\Zp[\Gcal_n]$. Therefore
\[
\varpi_n(\mathcal N_n)=\Princ_p(Y_n),
\]
and the result follows.
\end{proof}

\begin{definition}\label{def:Pic-Iwasawa}
The equivariant Picard Iwasawa module attached to $Y_\infty/X$ is
\[
\Pic_{\LambdaG}(Y_\infty)
:=\varprojlim_n\Pic_p(Y_n),
\]
where the transition maps are induced by the covering maps $Y_{n+1}\to Y_n$.
\end{definition}

We now pass from the finite-level description in Proposition~\ref{prop:finite-codescent} to the inverse limit. The following theorem identifies the equivariant Picard Iwasawa module with a quotient of the divisor module at the top of the tower.

\begin{theorem}\label{thm:Pic-quotient}
There is a natural isomorphism of $\LambdaG$-modules
\begin{equation}\label{eq:Pic-quotient}
\Pic_{\LambdaG}(Y_\infty)
\xrightarrow{\sim}
\frac{\Div_{\LambdaG}(Y_\infty)}{\Princ^A+\Princ^B}.
\end{equation}
\end{theorem}

\begin{proof}
For $n$ sufficiently large, Proposition~\ref{prop:finite-codescent}
gives an isomorphism
\[
\frac{\Div_{\LambdaG}(Y_\infty)}{\mathcal N_n}
\xrightarrow{\sim}
\Pic_p(Y_n).
\]
These isomorphisms are compatible with the transition maps. More precisely,
the natural projection
\[
\frac{\Div_{\LambdaG}(Y_\infty)}{\mathcal N_{n+1}}
\longrightarrow
\frac{\Div_{\LambdaG}(Y_\infty)}{\mathcal N_n}
\]
corresponds, under the above isomorphisms, to the transition map
\[
\Pic_p(Y_{n+1})\longrightarrow\Pic_p(Y_n).
\]
Therefore,
\[
\Pic_{\LambdaG}(Y_\infty)
\simeq
\varprojlim_n
\frac{\Div_{\LambdaG}(Y_\infty)}{\mathcal N_n}.
\]
Since
\[
\Princ^A+\Princ^B\subseteq \mathcal N_n
\]
for every sufficiently large $n$, we obtain natural maps
\[
\frac{\Div_{\LambdaG}(Y_\infty)}
     {\Princ^A+\Princ^B}
\longrightarrow
\frac{\Div_{\LambdaG}(Y_\infty)}{\mathcal N_n},
\]
which induce a surjective $\LambdaG$-module morphism
\begin{equation}\label{eq:natural-map-Pic}
\frac{\Div_{\LambdaG}(Y_\infty)}
     {\Princ^A+\Princ^B}
\longrightarrow
\varprojlim_n
\frac{\Div_{\LambdaG}(Y_\infty)}{\mathcal N_n}.
\end{equation}
It remains to show that this morphism is injective. We first claim that
\begin{equation}\label{eq:intersection-Nn}
\bigcap_{n\gg 0}\mathcal N_n
=
\Princ^A+\Princ^B.
\end{equation}
The inclusion from right to left is clear from the definition of
$\mathcal N_n$. Conversely, let
\[
D\in\bigcap_{n\gg0}\mathcal N_n.
\]
For every sufficiently large $n$, we may write
\[
D=D_n'+P_n,
\]
where
\[
D_n'
\in
\omega_n\Div_{\LambdaG}(Y_\infty)+\Princ^{C,n}
\]
and
\[
P_n\in\Princ^A+\Princ^B.
\]
By definition,
\[
\Princ^{C,n}
=
\sum_{v\in C}
\LambdaG\,
\Norm{\wtI_v^{\,0}}\omega_{n,k_v}P_{v,\infty}.
\]
Since $\omega_n\to0$ and $\omega_{n,k_v}\to0$ in $\LambdaG$ as
$n\to\infty$, and since $\Div_{\LambdaG}(Y_\infty)$ and $\LambdaG$ are
compact, we have
\[
D_n'\longrightarrow 0.
\]
Consequently,
\[
P_n\longrightarrow D.
\]
On the other hand, $\Princ^A+\Princ^B$ is finitely generated over
$\LambdaG$, and hence is compact. In particular, it is closed in
$\Div_{\LambdaG}(Y_\infty)$. It follows that
\[
D\in\Princ^A+\Princ^B,
\]
which proves \eqref{eq:intersection-Nn}.

The equality \eqref{eq:intersection-Nn} shows that the morphism
\eqref{eq:natural-map-Pic} is injective. This completes the proof.
\end{proof}

\section{Equivariant Artin--Ihara $L$-elements}

In this section, we construct the equivariant Ihara \(L\)-element associated with \(Y_\infty/X\), starting from the determinant formula for finite abelian branched covers. We then introduce the modified \(L\)-polynomials whose values at \(u=1\) occur in the equivariant main conjecture.

\subsection{Finite abelian branched covers}

Let
\[
Z=X(H,\mathcal J,\beta)\longrightarrow X
\]
be a finite abelian branched cover of finite connected graphs for the ramification data and the voltage assignment
\[
\mathcal J=\{(v,J_v):v\in V_X\},
\qquad
\beta:E_X\longrightarrow H.
\]
Thus $H$ acts on $Z$ without inversion and freely on directed edges,
and the stabilizer of a vertex above $v$ is $J_v$. Set
\[
\ell=|V_X|,
\qquad
m=|O_X|=\frac{|E_X|}{2},
\]
so that $\chi(X)=\ell-m$.

We use the Ihara $L$-functions and the equivariant formalism developed
in \cite[Sections~4--5]{GV26}. We first recall the character notation and the determinant description
needed below.\\

Let $\widehat H=\operatorname{Hom}(H,\C^\times)$. For
$\psi\in\widehat H$, we also write $\psi$ for the induced
$\C$-algebra morphism
\[
\psi:\C[H]\longrightarrow\C
\]
and for its coefficientwise extensions to polynomials, power series,
and matrices. Let
\[
e_\psi
=
\frac{1}{|H|}
\sum_{\sigma\in H}\psi(\sigma)\sigma^{-1}
\]
be the associated primitive idempotent. In particular,
\[
\C[H]
=
\bigoplus_{\psi\in\widehat H}\C e_\psi,
\qquad
\psi(e_\varphi)
=
\begin{cases}
1,&\psi=\varphi,\\
0,&\psi\ne\varphi.
\end{cases}
\]
Consider the $\C[H]$-module
\[
M_Z=\Div(Z)\otimes_{\Z}\C.
\]
Its adjacency and valency operators are given by
\[
\mathsf{A}_Z(w)
=
\sum_{\epsilon\in E_{Z,w}}t(\epsilon),
\qquad
\mathsf{D}_Z(w)
=
\val_Z(w)\,w.
\]
Both operators are $H$-equivariant. They therefore preserve
\[
M_{Z,\psi}=e_\psi M_Z.
\]
Write $\mathsf{A}_{Z,\psi}$ and $\mathsf{D}_{Z,\psi}$ for their
restrictions to this space, and let $I_\psi$ denote its identity
operator. Following the notation of \cite[Section~5.3]{GV26}, put
\[
h(u,\psi)
=
det_{\C[u]}
\left(
I_\psi-\mathsf{A}_{Z,\psi}u
+(\mathsf{D}_{Z,\psi}-I_\psi)u^2
\right).
\]
Here the determinant is taken on
$M_{Z,\psi}\otimes_{\C}\C[u]$, and a determinant on the zero module
is understood to be $1$.

For $\psi\in\widehat H$, define
\[
r_0(\psi)
=
\bigl|\{v\in V_X:J_v\not\subseteq\ker(\psi)\}\bigr|.
\]
The permutation-module decomposition
\[
M_Z
\simeq
\bigoplus_{v\in V_X}\C[H/J_v]
\]
gives
\[
\dim_{\C}M_{Z,\psi}
=
\ell-r_0(\psi).
\]
Since the action on directed edges is free and without inversion,
the corresponding character component of the oriented-edge module
has dimension $m$. Thus, with the conventions of \cite{GV26},
the Ihara $L$-function associated with $\psi$ satisfies
\[
L(u,\psi,Z/X)^{-1}
=
(1-u^2)^{m-\ell+r_0(\psi)}h(u,\psi).
\]
This is the characterwise determinant description established in
\cite[Section~5.3]{GV26}. All expressions are interpreted as formal
power series at $u=0$.

We define the equivariant $L$-element by
\[
\Theta_{Z/X}(u)
=
\sum_{\psi\in\widehat H}
L(u,\psi,Z/X)^{-1}e_\psi
\in 1+u\C[H][[u]].
\]
Equivalently, it is characterized by
\begin{equation}\label{eq:theta-interpolation}
\psi\bigl(\Theta_{Z/X}(u)\bigr)
=
L(u,\psi,Z/X)^{-1}.
\end{equation}
In the notation of \cite[Section~4.4]{GV26}, this means that
\[
\Theta_{Z/X}(u)=\theta_Z(u)^{-1},
\]
where $\theta_Z(u)$ denotes the equivariant Ihara zeta function.

We now introduce the matrices that express this element integrally.
Fix an ordering
\[
V_X=\{v_1,\ldots,v_\ell\},
\]
and let
\[
D
=
\operatorname{diag}
\bigl(\val_X(v_1),\ldots,\val_X(v_\ell)\bigr).
\]
Recall that
\[
\Norm{J_v}
=
\sum_{\sigma\in J_v}\sigma,
\]
and set
\[
C_{\mathcal J}
=
\operatorname{diag}
\bigl(\Norm{J_{v_1}},\ldots,\Norm{J_{v_\ell}}\bigr).
\]
Define the voltage adjacency matrix
\[
A_\beta=(a_{ij}(\beta))\in M_\ell(\Z[H])
\]
by
\begin{equation}\label{eq:voltage-adjacency}
a_{ij}(\beta)
=
\sum_{\substack{e\in E_X\\
o(e)=v_i,\ t(e)=v_j}}
\beta(e).
\end{equation}

With this convention, the row index records the origin of an edge
and the column index records its terminus. Consequently, $A_\beta$
is the transpose of the voltage adjacency matrix used in
\cite[Definition~6.4]{GV26}. This correspondence will allow us to
apply the formula on page~43 of that paper directly.

Put
\[
\xi_Z(u)
=
\det\left(
I_\ell-C_{\mathcal J}A_\beta u
+(C_{\mathcal J}D-I_\ell)u^2
\right)
\in 1+u\Z[H][u].
\]
Taking the transpose inside the determinant gives
\[
\xi_Z(u)
=
\det\left(
I_\ell-A_\beta^{\mathsf t}C_{\mathcal J}u
+(DC_{\mathcal J}-I_\ell)u^2
\right).
\]
Thus $\xi_Z(u)$ is precisely the polynomial denoted by $\xi_Y(u)$
in \cite[Definition~6.4]{GV26}, with $Y=Z$.

\begin{proposition}\label{prop:finite-theta-det}
With the notation above, one has
\begin{equation}\label{eq:finite-theta-det}
\Theta_{Z/X}(u)
=
(1-u^2)^{m-\ell}
\det\left(
I_\ell-C_{\mathcal J}A_\beta u
+(C_{\mathcal J}D-I_\ell)u^2
\right).
\end{equation}
In particular,
\[
\Theta_{Z/X}(u)\in 1+u\Z[H][[u]],
\]
and its coefficients have a canonical image in $\Zp[H]$ for every
prime $p$.
\end{proposition}

\begin{proof}
Fix $\psi\in\widehat H$, and let
\[
z(u,\psi)=\psi(\xi_Z(u)),
\]
as in \cite[Definition~6.4]{GV26}. Under the matrix correspondence
explained above, the formula on page~43
of \cite{GV26} reads
\[
z(u,\psi)
=
\det\left(
I_\ell
-\psi(C_{\mathcal J})\psi(A_\beta)u
+\bigl(\psi(C_{\mathcal J})D-I_\ell\bigr)u^2
\right).
\]
Here $\psi(D)=D$, since the entries of $D$ are integers.

The calculation immediately preceding that formula establishes
\[
L(u,\psi,Z/X)^{-1}
=
(1-u^2)^{-\chi(X)}z(u,\psi).
\]
It follows that
\[
\begin{aligned}
\psi\bigl(\Theta_{Z/X}(u)\bigr)
&=
L(u,\psi,Z/X)^{-1}\\
&=
(1-u^2)^{m-\ell}\psi(\xi_Z(u))\\
&=
\psi\bigl((1-u^2)^{m-\ell}\xi_Z(u)\bigr).
\end{aligned}
\]
Since the characters of $H$ separate elements of $\C[H]$,
applying this coefficientwise proves
\[
\Theta_{Z/X}(u)
=
(1-u^2)^{m-\ell}\xi_Z(u).
\]

Finally, $\xi_Z(u)\in 1+u\Z[H][[u]]$, and
\[
(1-u^2)^k\in1+u\Z[H][[u]]
\qquad\text{for every }k\in\Z.
\]
Consequently, the displayed identity proves the asserted integrality
even when $m-\ell<0$. The statement over $\Zp[H]$ follows by applying
the natural coefficient map $\Z[H]\to\Zp[H]$.
\end{proof}

\subsection{The equivariant $L$-element at the top}

Return to the tower $Y_\infty/X$ from Section 3.
Using the vertex ordering $V_X=\{v_1,\ldots,v_\ell\}$ fixed in
Section~4.1, define the voltage adjacency matrix
\[
A_{\wtalpha}=(\widetilde a_{ij})
\in M_\ell(\LambdaG),
\qquad
\widetilde a_{ij}
=
\sum_{\substack{e\in E_X\\
o(e)=v_i,\ t(e)=v_j}}
\wtalpha(e).
\]
Its image under the natural coefficient map
$\LambdaG\longrightarrow\Zp[\Gcal_n]$ is the finite-level voltage
adjacency matrix $A_{\alpha_n}$. For $n\ge0$, let
\[
J_n=\bigl(\Norm{I_{v,n}}:v\in C\bigr)\subseteq\Zp[\Gcal_n].
\]
The natural maps $\Gcal_{n+1}\twoheadrightarrow\Gcal_n$ induce compatible maps
\[
\Zp[\Gcal_{n+1}]/J_{n+1}\longrightarrow\Zp[\Gcal_n]/J_n.
\]

\begin{lemma}\label{lem:theta-compatible}
The images of the finite-level equivariant $L$-elements are compatible:
\[
\Theta_{Y_{n+1}/X}(u)\longmapsto\Theta_{Y_n/X}(u)
\quad\text{in }(\Zp[\Gcal_n]/J_n)[[u]].
\]
Moreover,
\[
\varprojlim_n J_n=0.
\]
\end{lemma}

\begin{proof}
Apply Proposition \ref{prop:finite-theta-det} at levels $n+1$ and $n$. The voltage matrices and degree matrices are compatible under projection. For $v\notin C$, the norm elements are compatible because the maps $\wtI_v\to I_{v,n}$ are eventually isomorphisms onto their images. For $v\in C$, the discrepancy is killed modulo $J_n$. This gives the first assertion.

For the second assertion, the case $C=\varnothing$ is immediate,
since then $J_n=0$ for every $n$. Assume $C\ne\varnothing$, and put
\[
n_0=\max_{v\in C}k_v,
\]
Write
\[
R_n=\Zp[\Gcal_n],
\]
and denote the natural transition maps by
\[
\rho_{m,n}:R_m\longrightarrow R_n
\qquad(m\ge n).
\]
For $v\in C$ and $n\ge n_0$, we have
$\Gamma_n\subseteq\wtI_v$. Since $|\Gamma_n/\Gamma_{n+1}|=p$, every element of $I_{v,n}$
has exactly $p$ preimages in $I_{v,n+1}$. Therefore,
\[
\rho_{n+1,n}\bigl(\Norm{I_{v,n+1}}\bigr)
=
p\Norm{I_{v,n}}
\qquad(v\in C,\ n\ge n_0).
\]

Each $R_n$, endowed with its $p$-adic topology, is a compact
Noetherian commutative ring, since it is finite free over $\Zp$.
The transition maps are surjective, and $p^r\to0$ in every $R_n$.
Moreover, the ideals
\[
J_n=\bigl(\Norm{I_{v,n}}:v\in C\bigr)
\]
are integral and hence have bounded denominators. We may therefore
apply \cite[Lemma~2.7]{G25} to the system indexed by $n\ge n_0$,
taking its first ideal to be our $J_n$, its second ideal to be
$R_n$, and its family of compatible generators to be empty.
The preceding identity shows that all generators belong to the
family that is multiplied by $p$ under transition. The lemma gives
\[
\varprojlim_{n\ge n_0}J_n=0.
\]
Discarding finitely many initial terms does not change the inverse
limit, so
\[
\varprojlim_n J_n=0.
\]
\end{proof}

The previous lemma allows us to make the following definition.

\begin{definition}\label{def:theta-infty}
The equivariant Artin--Ihara $L$-element at the top of the tower is the unique element
\[
\Theta_{Y_\infty/X}(u)\in\LambdaG[[u]]
\]
whose image in $(\Zp[\Gcal_n]/J_n)[[u]]$ is the image of $\Theta_{Y_n/X}(u)$ for every $n$.
\end{definition}

\subsection{Removing the infinitely ramified vertices}

Let $A_{\wtalpha}^{C}$ be the matrix obtained from the voltage adjacency matrix $A_{\wtalpha}$ by deleting the rows and columns corresponding to vertices in $C$. Similarly, let $D^{C}$ be the corresponding truncated degree matrix. Since $\wtI_v$ is finite for $v\notin C$, define
\[
C_{\wtI}^{C}=\operatorname{diag}(\Norm{\wtI_v})_{v\notin C}.
\]
Thus all three matrices have size $(\ell-|C|)\times(\ell-|C|)$ and entries in $\LambdaG$.

\begin{definition}\label{def:theta-star}
We define the modified equivariant $L$-polynomial by
\begin{equation}\label{eq:theta-star}
\ThetaStar_{Y_\infty/X}(u)
=
\det\!\left(
I_{\ell-|C|}
-C_{\wtI}^{C}A_{\wtalpha}^{C}u
+(C_{\wtI}^{C}D^{C}-I_{\ell-|C|})u^2
\right).
\end{equation}
\end{definition}

\begin{theorem}\label{thm:theta-factorization}
One has
\begin{equation}\label{eq:theta-factorization}
\Theta_{Y_\infty/X}(u)
=(1-u^2)^{m-\ell+|C|}\ThetaStar_{Y_\infty/X}(u).
\end{equation}
Moreover, $\ThetaStar_{Y_\infty/X}(u)\in\LambdaG[u]$ and its value at $u=1$ is
\begin{equation}\label{eq:theta-star-one}
\ThetaStar_{Y_\infty/X}(1)
=
\det\!\left(C_{\wtI}^{C}(D^{C}-A_{\wtalpha}^{C})\right).
\end{equation}
\end{theorem}

\begin{proof}
At level $n$, Proposition \ref{prop:finite-theta-det} gives
\[
\Theta_{Y_n/X}(u)
=(1-u^2)^{m-\ell}
\det\!\left(I_\ell-C_nA_nu+(C_nD-I_\ell)u^2\right),
\]
where $C_n=\operatorname{diag}(\Norm{I_{v,n}})$. Work modulo $J_n$. For every $v\in C$, the corresponding diagonal entry of $C_n$ vanishes. After ordering the vertices so that those in $C$ come last, the matrix becomes block triangular of the form
\[
\begin{pmatrix}
K_n^{C}(u)&\ast\\
0&(1-u^2)I_{|C|}
\end{pmatrix},
\]
where
\[
K_n^{C}(u)=
I_{\ell-|C|}-C_n^{C}A_n^{C}u+(C_n^{C}D^{C}-I_{\ell-|C|})u^2.
\]
Therefore, in $(\Zp[\Gcal_n]/J_n)[[u]]$,
\[
\Theta_{Y_n/X}(u)
=(1-u^2)^{m-\ell+|C|}\det K_n^{C}(u).
\]
For $v\notin C$, the natural map $\wtI_v\to I_{v,n}$ is eventually an isomorphism, and the truncated matrices are compatible under projection. Passing to the inverse limit and using Definition \ref{def:theta-infty} proves \eqref{eq:theta-factorization}. Setting $u=1$ in \eqref{eq:theta-star} gives \eqref{eq:theta-star-one}.
\end{proof}

For later use, we modify the ramification at a chosen subset of the finite-inertia vertices. If $B'\subseteq B$, set
\[
\wtI_v^{B'}=
\begin{cases}
\Gcal,&v\in B',\\
\wtI_v,&v\notin B',
\end{cases}
\]
and put
\[
\wtIcal^{B'}=\{(v,\wtI_v^{B'}):v\in V_X\},
\qquad
Y_\infty^{B'}=X(\Gcal,\wtIcal^{B'},\wtalpha).
\]
Thus the vertices in $B'$ are made totally ramified in $Y_\infty^{B'}/X$. Its set of infinitely ramified vertices is $B'\cup C$. We write
\[
\ThetaStar_{Y_\infty^{B'}/X}(u)
\]
for the polynomial of Definition \ref{def:theta-star} attached to this modified tower. Explicitly, if $A_{\wtalpha}^{B'}$ and $D^{B'}$ denote the matrices obtained by deleting rows and columns indexed by $B'\cup C$, and
\[
C_{\wtI}^{B'}=\operatorname{diag}(\Norm{\wtI_v})_{v\in A\cup(B\setminus B')},
\]
then
\begin{equation}\label{eq:theta-Bprime-one}
\ThetaStar_{Y_\infty^{B'}/X}(1)
=
\det\!\left(C_{\wtI}^{B'}(D^{B'}-A_{\wtalpha}^{B'})\right).
\end{equation}

\section{A presentation and the equivariant main conjecture}

We now construct a free presentation of the equivariant Picard Iwasawa module and compute its Fitting ideal using the determinant formulas of Section 4. This proves the equivariant main conjecture, from which we recover the main conjecture for branched \(\mathbb Z_p\)-towers as a special case.

\subsection{A free presentation}

For every $v\in B\cup C$, fix topological generators
\[
g_{v,1},\ldots,g_{v,r_v}
\]
of $\wtI_v$. Consider the free $\LambdaG$-module
\[
\mathcal D^{\nr}
:=\Div_{\LambdaG}(Y_\infty^{\nr})
=\bigoplus_{v\in V_X}\LambdaG\,w_{v,\infty}^{\nr}.
\]
Define a $\LambdaG$-linear map
\begin{equation}\label{eq:Delta-map}
\Delta:
\left(\bigoplus_{v\in A\cup B}\LambdaG e_v\right)
\oplus
\left(\bigoplus_{v\in B\cup C}\bigoplus_{j=1}^{r_v}\LambdaG f_{v,j}\right)
\longrightarrow \mathcal D^{\nr}
\end{equation}
by
\begin{align}
\Delta(e_v)&=\Norm{\wtI_v}\,L_{Y_\infty^{\nr}}(w_{v,\infty}^{\nr}),
&&v\in A\cup B,\label{eq:Delta-first}\\
\Delta(f_{v,j})&=(g_{v,j}-1)w_{v,\infty}^{\nr},
&&v\in B\cup C.\label{eq:Delta-second}
\end{align}
For $v\in A$ we use the convention $\Norm{\wtI_v}=1$.

\begin{theorem}\label{thm:presentation}
The map \eqref{eq:Delta-map} is a presentation of the equivariant Picard Iwasawa module. More precisely, there is an exact sequence
\begin{equation}\label{eq:presentation-sequence}
\left(\bigoplus_{v\in A\cup B}\LambdaG e_v\right)
\oplus
\left(\bigoplus_{v\in B\cup C}\bigoplus_{j=1}^{r_v}\LambdaG f_{v,j}\right)
\xrightarrow{\Delta}
\mathcal D^{\nr}
\longrightarrow
\Pic_{\LambdaG}(Y_\infty)
\longrightarrow0.
\end{equation}
\end{theorem}

\begin{proof}
The morphism $\iota:Y_\infty^{\nr}\to Y_\infty$ induces a surjective $\LambdaG$-module map
\[
\iota_*:\mathcal D^{\nr}\longrightarrow\Div_{\LambdaG}(Y_\infty).
\]
On the $v$-component, this is the natural quotient map
\[
\LambdaG\longrightarrow\Zp[[\Gcal/\wtI_v]].
\]
Therefore
\begin{equation}\label{eq:ker-iota}
\kerop(\iota_*)
=
\bigoplus_{v\in B\cup C}\aug{\wtI_v}\,w_{v,\infty}^{\nr}
=
\sum_{v\in B\cup C}\sum_{j=1}^{r_v}
\LambdaG(g_{v,j}-1)w_{v,\infty}^{\nr}.
\end{equation}
This is exactly the image under $\Delta$ of the second family of summands in \eqref{eq:Delta-map}.

For $v\in A\cup B$, applying $\iota_*$ to \eqref{eq:Delta-first} gives
\[
\iota_*\Delta(e_v)=\Norm{\wtI_v}P_{v,\infty}.
\]
As $v$ ranges through $A\cup B$, these elements generate $\Princ^A+\Princ^B$. Hence the cokernel of $\Delta$ is
\[
\frac{\Div_{\LambdaG}(Y_\infty)}{\Princ^A+\Princ^B},
\]
which is $\Pic_{\LambdaG}(Y_\infty)$ by Theorem \ref{thm:Pic-quotient}.
\end{proof}

\subsection{The Fitting ideal}

We can now state and prove the equivariant main conjecture for the tower $Y_\infty/X$.

\begin{theorem}[Equivariant main conjecture]\label{thm:EMC}
With the notation above,
\begin{equation}\label{eq:EMC}
\Fitt_{\LambdaG}\!\left(\Pic_{\LambdaG}(Y_\infty)\right)
=
\left(
\ThetaStar_{Y_\infty^{B'}/X}(1)
\prod_{v\in B'\cup C}\aug{\wtI_v}
\ ;\ B'\subseteq B
\right).
\end{equation}
Here the right-hand side denotes the ideal generated, as $B'$ ranges over all subsets of $B$, by the products of the element $\ThetaStar_{Y_\infty^{B'}/X}(1)$ with the indicated augmentation ideals.
\end{theorem}

\begin{proof}
Let $r=|V_X|$. By Theorem \ref{thm:presentation}, the Fitting ideal on the left of \eqref{eq:EMC} is generated by the $r\times r$ minors of a matrix for $\Delta$.

Order the rows by the vertices of $X$. For each $v\in A\cup B$, there is a distinguished ``Laplacian column'' coming from $e_v$. For each $v\in B\cup C$, there are $r_v$ ``augmentation columns'' coming from the elements $f_{v,j}$; every such column is supported only in the row indexed by $v$, with entry $g_{v,j}-1$.

Suppose first that an $r\times r$ minor contains both the Laplacian column attached to some $v\in B$ and an augmentation column attached to the same vertex. Expanding along the latter column shows that the determinant contains a factor
\[
\Norm{\wtI_v}(g_{v,j}-1)=0.
\]
Thus such a minor vanishes. A nonzero minor also cannot contain two augmentation columns attached to the same vertex, since those columns are supported in the same row. It follows that every potentially nonzero maximal minor is obtained by choosing a subset $B'\subseteq B$, taking Laplacian columns for all vertices in $A\cup(B\setminus B')$, and taking exactly one augmentation column for each vertex in $B'\cup C$.

Fix $B'\subseteq B$ and, for each $v\in B'\cup C$, choose an index $j_v\in\{1,\ldots,r_v\}$. After reordering rows and columns, the corresponding submatrix has block form
\begin{equation}\label{eq:block-minor}
P=
\begin{pmatrix}
C_{\wtI}^{B'}(D^{B'}-A_{\wtalpha}^{B'})&0\\
\ast&\operatorname{diag}(g_{v,j_v}-1)_{v\in B'\cup C}
\end{pmatrix}.
\end{equation}
Indeed, the upper-left block is obtained from the weighted Laplacian columns after deleting the rows and columns indexed by $B'\cup C$. Therefore, by \eqref{eq:theta-Bprime-one},
\[
\det\!\left(C_{\wtI}^{B'}(D^{B'}-A_{\wtalpha}^{B'})\right)
=\ThetaStar_{Y_\infty^{B'}/X}(1).
\]
Consequently,
\begin{equation}\label{eq:minor-value}
\det(P)
=
\ThetaStar_{Y_\infty^{B'}/X}(1)
\prod_{v\in B'\cup C}(g_{v,j_v}-1).
\end{equation}
As the indices $j_v$ vary, the products on the right generate
\[
\prod_{v\in B'\cup C}\aug{\wtI_v}
\]
by \eqref{eq:augmentation-generators}. Finally, allowing $B'$ to vary over all subsets of $B$ gives exactly the ideal on the right-hand side of \eqref{eq:EMC}.
\end{proof}

\begin{remark}\label{rem:ramification-meaning}
The three kinds of vertices have distinct roles in \eqref{eq:EMC}. Vertices in $A$ contribute only through the modified $L$-value. Vertices in $C$ are forced to contribute their augmentation ideals in every generator of the Fitting ideal. Vertices in $B$ interpolate between the two possibilities: for each generator one either retains the finite norm factor inside the determinant or replaces that vertex by its augmentation ideal. This is the source of the sum over subsets $B'\subseteq B$.
\end{remark}

\begin{remark}\label{rem:arithmetic-analogy}
Theorem~\ref{thm:EMC} is a graph-theoretic analogue of the integral
Fitting-ideal formulations of equivariant main conjectures in number
theory. The module
$\operatorname{Pic}_{\Lambda(\mathcal G)}(Y_\infty)$
plays the role of an Iwasawa module obtained from class groups, while
the modified equivariant Ihara elements
$\Theta^*_{Y_\infty^{B'}/X}(1)$
play the role of equivariant $p$-adic $L$-elements. The broader
formulation of main conjectures through Fitting ideals is exemplified
by Greither--Popescu~\cite[Theorem~5.6]{GrP15}, proved under the
vanishing of the relevant $\mu$-invariant.

A closer comparison, accounting for the local factors, is provided by
Gambheera~\cite[Theorem~1.3]{G25} and
Gambheera--Popescu~\cite[Theorem~4.9]{GP25RW}.
These results compute the Fitting ideal of the minus part of an
inverse limit of modified class groups in terms of an equivariant
$p$-adic $L$-element and explicit local ideals determined by inertia
and decomposition groups. In Theorem~\ref{thm:EMC}, the finite inertia
norms occurring in the determinants and the augmentation ideals
$\Delta_{\widetilde I_v}$ have an analogous role: they incorporate
the ramification data into the formula for the Fitting ideal.
\end{remark}
\begin{remark}
If $Y_\infty/X$ is unramified, then $B=C=\varnothing$, and
Theorem~\ref{thm:EMC} recovers the Picard-group Fitting formula of
Kataoka~\cite[Proposition~5.10]{Kataoka}, with
$\Theta^*_{Y_\infty/X}(1)$ equal to his determinant
$Z_{X,\Gamma}$.
\end{remark}

\subsection{Recovery of the one-variable main conjecture}

We now recover the main conjecture for a single branched $\Zp$-tower from \cite{GV24,GV26}. Let
\[
X=X_0\longleftarrow X_1\longleftarrow\cdots\longleftarrow X_\infty
\]
be a branched $\Zp$-tower arising from a voltage assignment $\alpha:E_X\to\Gamma\simeq\Zp$, and assume every $X_n$ is connected. Put
\[
\Lambda=\Zp[[\Gamma]].
\]
If $v$ is ramified, write
\[
I_v=\Gamma^{p^{k_v}}.
\]

\begin{corollary}[The classical graph-theoretic main conjecture]\label{cor:classical-MC}
With the notation above,
\begin{equation}\label{eq:classical-MC}
\charideal_\Lambda\!\left(\Pic_\Lambda(X_\infty)\right)
=
\left(
\ThetaStar_{X_\infty/X}(1)
\prod_{v\in V_X^{\ram}}(\gamma^{p^{k_v}}-1)
\right).
\end{equation}
\end{corollary}

\begin{proof}
Apply Theorem \ref{thm:EMC} with $G=1$ and $\Gcal=\Gamma$. Since $\Gamma\simeq\Zp$ has no nontrivial finite subgroup, $B=\varnothing$ and $C=V_X^{\ram}$. Moreover,
\[
\aug{I_v}=(\gamma^{p^{k_v}}-1).
\]
Thus Theorem \ref{thm:EMC} gives the right-hand side of \eqref{eq:classical-MC} as the Fitting ideal of $\Pic_\Lambda(X_\infty)$.

In this case every nontrivial inertia group is procyclic, so the presentation in Theorem \ref{thm:presentation} is square. The module is torsion over the regular local domain $\Lambda$, and its zeroth Fitting ideal agrees with its characteristic ideal. This proves the formula.
\end{proof}

\begin{remark}
Under a choice of topological generator $\gamma$ and the identification $\Lambda\simeq\Zp[[T]]$ given by $\gamma\mapsto1+T$, the factors $\gamma^{p^{k_v}}-1$ become
\[
(1+T)^{p^{k_v}}-1.
\]
Together with the relation between $\ThetaStar_{X_\infty/X}(1)$ and the power series denoted $g(T)$ in \cite[Section 6]{GV26}, Corollary \ref{cor:classical-MC} recovers the main-conjecture formula of \cite[Theorems 5.9 and 6.1]{GV24} and its Artin--Ihara formulation in \cite[Theorem 6.14]{GV26}.
\end{remark}

\section{Finite-level applications}

We now turn to finite abelian branched covers and establish a Fitting-ideal formula for their \(p\)-adic Picard groups. We then derive annihilation consequences and discuss the parallels with Kurihara’s formula and Brumer–Stark-type statements in number theory.

\subsection{A finite-level Fitting formula}

Let
\[
Y=X(G,\Ical,\alpha)\longrightarrow X
\]
be a finite abelian branched cover of finite connected graphs. Put
\[
V_X^{\ram}=\{v\in V_X:I_v\ne1\}.
\]
For $v\in V_X^{\ram}$, fix generators $g_{v,1},\ldots,g_{v,r_v}$ of $I_v$. If $B'\subseteq V_X^{\ram}$, define a modified inertia family
\[
I_v^{B'}=
\begin{cases}
G,&v\in B',\\
I_v,&v\notin B',
\end{cases}
\]
and let
\[
Y^{B'}=X(G,\Ical^{B'},\alpha).
\]
Thus the vertices in $B'$ are made totally ramified.

Delete from $A_\alpha$ and $D$ the rows and columns indexed by $B'$, and denote the resulting matrices by $A_\alpha^{B'}$ and $D^{B'}$. Set
\[
C_I^{B'}=\operatorname{diag}(\Norm{I_v})_{v\notin B'}.
\]
We define the modified equivariant $L$-polynomial
\begin{equation}\label{eq:finite-theta-star}
\ThetaStar_{Y^{B'}/X}(u)
=
\det\!\left(
I_{\ell-|B'|}
-C_I^{B'}A_\alpha^{B'}u
+(C_I^{B'}D^{B'}-I_{\ell-|B'|})u^2
\right).
\end{equation}
This definition is intrinsic to the finite cover. If $Y/X$ is embedded as the bottom layer of a tower as in Section 3, then \eqref{eq:finite-theta-star} is the image of $\ThetaStar_{Y_\infty^{B'}/X}(u)$ under $\LambdaG\to\Zp[G]$. In particular it is independent of the chosen tower.

\begin{theorem}[Finite-level equivariant Fitting formula]\label{thm:finite-Fitt}
One has
\begin{equation}\label{eq:finite-Fitt}
\Fitt_{\Zp[G]}\!\left(\Pic_p(Y)\right)
=
\left(
\ThetaStar_{Y^{B'}/X}(1)
\prod_{v\in B'}\aug{I_v}
\ ;\ B'\subseteq V_X^{\ram}
\right).
\end{equation}
\end{theorem}

\begin{proof}
The proof is the finite-level analogue of Theorems \ref{thm:presentation} and \ref{thm:EMC}. Let
\[
Y^{\nr}=X(G,\{(v,1):v\in V_X\},\alpha)
\]
and write $w_v^{\nr}=(v,1)$. Consider
\[
\Delta_Y:
\left(\bigoplus_{v\in V_X}\Zp[G]e_v\right)
\oplus
\left(\bigoplus_{v\in V_X^{\ram}}\bigoplus_{j=1}^{r_v}\Zp[G]f_{v,j}\right)
\longrightarrow
\Div_{\Zp[G]}(Y^{\nr})
\]
defined by
\begin{align*}
\Delta_Y(e_v)&=\Norm{I_v}L_{Y^{\nr}}(w_v^{\nr}),\\
\Delta_Y(f_{v,j})&=(g_{v,j}-1)w_v^{\nr}.
\end{align*}
The natural morphism $Y^{\nr}\to Y$ shows exactly as in the proof of Theorem \ref{thm:presentation} that
\[
\coker(\Delta_Y)\simeq\Pic_p(Y).
\]
Hence $\Fitt_{\Zp[G]}(\Pic_p(Y))$ is generated by the maximal minors of $\Delta_Y$.

As in the proof of Theorem \ref{thm:EMC}, every nonzero maximal minor is indexed by a subset $B'\subseteq V_X^{\ram}$ and a choice of one augmentation column at each $v\in B'$. After reordering, the corresponding matrix is block triangular with upper-left block
\[
C_I^{B'}(D^{B'}-A_\alpha^{B'})
\]
and lower-right block $\operatorname{diag}(g_{v,j_v}-1)_{v\in B'}$. Its determinant is
\[
\ThetaStar_{Y^{B'}/X}(1)\prod_{v\in B'}(g_{v,j_v}-1).
\]
Varying the chosen generators and then $B'$ gives \eqref{eq:finite-Fitt}.
\end{proof}

\begin{remark}\label{rem:all-generators-L}
Theorem \ref{thm:finite-Fitt} expresses every generator of the Fitting ideal in terms of a special value of a modified equivariant Artin--Ihara $L$-polynomial and augmentation ideals coming from ramification. In this sense the entire Fitting ideal is controlled by equivariant $L$-functions.
\end{remark}

\begin{remark}\label{rem:Kurihara-comparison}
Theorem~\ref{thm:finite-Fitt} gives a graph-theoretic analogue of
Kurihara's formula for the Fitting ideal of the minus part of the
Pontryagin dual of a ray class group in an abelian CM extension of a
totally real field. Proved away from $2$ by
Dasgupta--Kakde~\cite[Theorem~3.5]{DK23}, this formula identifies the
Fitting ideal with the Sinnott--Kurihara ideal, constructed from
Stickelberger elements and local factors at ramified primes.
See also Gambheera~\cite[Theorem~1.6]{G25}, which extends the formula
to allow tamely ramified primes in the smoothing set.

In the graph setting, the modified equivariant Ihara values
$\Theta^*_{Y^{B'}/X}(1)$ and the augmentation ideals $\Delta_{I_v}$
provide the corresponding special values and local factors.
The indexing by subsets of ramified vertices also parallels the
description of the Sinnott--Kurihara ideal by generators indexed
by subsets of ramified primes; compare
Dasgupta--Kakde~\cite[Lemma~3.4 and equation~(34)]{DK23}.
Both descriptions determine the Fitting ideal over the integral
group ring while retaining the local ramification data.
\end{remark}

\subsection{Annihilation consequences and the Brumer--Stark type statement}

The finite-level Fitting formula has the following immediate annihilation consequence.

\begin{corollary}[Brumer--Stark type annihilation]\label{cor:immediate-annihilation}
For every subset $B'\subseteq V_X^{\ram}$, one has
\[
\ThetaStar_{Y^{B'}/X}(1)\prod_{v\in B'}\aug{I_v}
\subseteq
\Ann_{\Zp[G]}\!\left(\Pic_p(Y)\right)
\subseteq
\Ann_{\Zp[G]}\!\left(\Pic_p^0(Y)\right).
\]
In particular,
\[
\ThetaStar_{Y/X}(1)\in
\Ann_{\Zp[G]}\!\left(\Pic_p(Y)\right).
\]
\end{corollary}

\begin{proof}
This follows immediately from Theorem \ref{thm:finite-Fitt} and the general inclusion
$\Fitt_R(M)\subseteq \Ann_R(M)$.
\end{proof}

\begin{remark}
When the cover $Y/X$ is unramified, $V_X^{\ram}=\varnothing$ and Theorem \ref{thm:finite-Fitt} reduces to the principal formula
\[
\Fitt_{\Zp[G]}(\Pic_p(Y))=(\ThetaStar_{Y/X}(1)),
\]
in agreement with the unramified results of Kataoka \cite{Kataoka}.
\end{remark}

\section{Examples}\label{sec:examples}

In this section, we calculate the equivariant Fitting ideals and the finite
Picard groups in two examples. As in \cite[Section~6]{GV24} and
\cite[Section~7]{GV26}, we use a power series variable to make the calculations
explicit. The first example has a vertex with nontrivial finite inertia. In
the second example, finite and infinite inertia occur simultaneously.

Throughout this section, let $p$ be an odd prime and put
\[
G=\langle\sigma\rangle\simeq\mathbb Z/p\mathbb Z,
\qquad \Gamma\simeq\mathbb Z_p,\qquad \Gcal=G\times\Gamma,
\qquad \Lambda=\mathbb Z_p[[\Gamma]].
\]
Fix a topological generator $\gamma$ of $\Gamma$. We identify
$\Lambda$ with $\mathbb Z_p[[T]]$ by $\gamma=1+T$, and write
\[
R=\LambdaG=\mathbb Z_p[G][[T]],\qquad
N=1+\sigma+\cdots+\sigma^{p-1}.
\]
In particular, $N\delta=0$ and $(N,\delta)=(p,\delta)$ in $R$.
For each of the towers below, write
\[
M=\Pic_{\LambdaG}(Y_\infty),\qquad
M^0=\varprojlim_n\Pic_p^0(Y_n).
\]
When discussing ordinary Iwasawa invariants, we regard these as
$\Lambda$-modules by restriction of scalars. The symbols $\sigma$, $\gamma$, $N$, and $T$ also denote their images
in the group rings $R_n=\mathbb Z_p[\Gcal_n]$ and $
\Lambda_n=\mathbb Z_p[\Gamma/\Gamma_n]$ whenever appropriate.

\subsection{A vertex with finite inertia}\label{subsec:example-finite-inertia}

Let $X$ have two vertices $v_1,v_2$, one loop $e$ at $v_1$, and one edge $f$
from $v_1$ to $v_2$. Choose the voltage assignment and inertia groups
\[
\wtalpha(e)=\gamma,\qquad \wtalpha(f)=1,\qquad
\wtI_{v_1}=1,\qquad \wtI_{v_2}=G.
\]
The voltages on the reversed edges are determined by inversion. Thus
\[
A=\{v_1\},\qquad B=\{v_2\},\qquad C=\varnothing.
\]
The graph $Y_n$ has vertices
\[
x_{h,i}\quad(h\in\mathbb Z/p\mathbb Z,\ i\in\mathbb Z/q\mathbb Z),
\qquad y_i\quad(i\in\mathbb Z/q\mathbb Z),
\]
with one edge from $x_{h,i}$ to $x_{h,i+1}$ and one edge from $x_{h,i}$ to
$y_i$ for each $(h,i)$. Here the first family consists of loops when $n=0$.
This description shows that every $Y_n$ is connected.

With the ordering $(v_1,v_2)$, the voltage adjacency and valency matrices are
\[
A_{\wtalpha}=
\begin{pmatrix}\gamma+\gamma^{-1}&1\\1&0\end{pmatrix},
\qquad D=\begin{pmatrix}3&0\\0&1\end{pmatrix}.
\]
Notice that 
\[
a:=3-\gamma-\gamma^{-1}
  =1-\frac{T^2}{1+T}
\]
is a unit in $\Lambda$ and in every $\Lambda_n$.
The presentation in Theorem~\ref{thm:presentation} becomes
\begin{equation}\label{eq:example-one-presentation}
R^3\xrightarrow{\displaystyle
\begin{pmatrix}a&-N&0\\-1&N&\sigma-1\end{pmatrix}}
R^2\longrightarrow M\longrightarrow0.
\end{equation}
Its maximal minors are $N(a-1)$, $a(\sigma-1)$, and $-N(\sigma-1)=0$. Consequently,
\begin{equation}\label{eq:example-one-fitting}
\Fitt_R(M)=(N(a-1),\sigma-1)=(\sigma-1,p(a-1))
          =(\sigma-1,pT^2).
\end{equation}
For the middle equality, use $N-p\in(\sigma-1)$.

The two subsets of $B$ give the special values
\[
\ThetaStar_{Y_\infty/X}(1)
=\det\begin{pmatrix}a&-1\\-N&N\end{pmatrix}
=N(a-1),
\qquad
\ThetaStar_{Y_\infty^{\{v_2\}}/X}(1)=a.
\]
Since $\aug{\wtI_{v_2}}=(\sigma-1)$, these recover
\eqref{eq:example-one-fitting} directly from Theorem~\ref{thm:EMC}.
The ideal in \eqref{eq:example-one-fitting} is non-principal. 

One can also determine $M$ itself. Let $x,y$ denote the images of the two
basis vectors of $R^2$ in \eqref{eq:example-one-presentation}. Its relations
are
\[
y=ax,\qquad N(y-x)=0,\qquad (\sigma-1) y=0.
\]
Since $a$ is a unit, these are equivalent to
$(\sigma-1) x=0$ and $pT^2x=0$. Hence
\[
M\simeq R/(\sigma-1,pT^2)
\quad\text{and, as a $\Lambda$-module,}\quad
M\simeq\Lambda/(pT^2).
\]
The degree map is induced by $T\mapsto0$ and sends $x$ to $1$.
Its kernel is $TM$, so
\begin{equation}\label{eq:example-one-degree-zero}
M^0\simeq\Lambda/(pT),\qquad
\charideal_\Lambda(M^0)=(pT).
\end{equation}
In particular, $\mu(M^0)=1$ and $\lambda(M^0)=1$.

At finite level the same calculation gives
\begin{equation}\label{eq:example-one-finite-pic}
\Pic_p(Y_n)\simeq\Lambda_n/(pT^2),
\end{equation}
with trivial $G$-action.For $n\geq1$, let $q=p^n$ and identify $\Lambda_n$ with
$\mathbb Z_p^q$ using the basis $1,\gamma,\ldots,\gamma^{q-1}$.
Multiplication by
$2-\gamma-\gamma^{-1}=-\gamma^{-1}T^2$ is then represented by the Laplacian $L_q$ of the cycle on $q$
vertices. Its Smith normal form is
$\operatorname{diag}
\bigl(\underbrace{1,\ldots,1}_{q-2},q,0\bigr)$.
Since $-\gamma^{-1}$ is a unit in $\Lambda_n$,
\eqref{eq:example-one-finite-pic} gives
\[
\Pic_p(Y_n)
\simeq
\operatorname{coker}_{\mathbb Z_p}(pL_q).
\]
Multiplying a Smith equivalence for $L_q$ by $p$ shows that
the Smith normal form of $pL_q$ is
\[
\operatorname{diag}
\bigl(\underbrace{p,\ldots,p}_{q-2},pq,0\bigr).
\]
Consequently, as $\mathbb Z_p$-modules,
\[
\Pic_p(Y_n)
\simeq
\mathbb Z_p
\oplus(\mathbb Z/p\mathbb Z)^{q-2}
\oplus\mathbb Z/p^{n+1}\mathbb Z
\]
and 
\[
\Pic_p^0(Y_n)
\simeq
(\mathbb Z/p\mathbb Z)^{q-2}
\oplus\mathbb Z/p^{n+1}\mathbb Z.
\]
For $n=0$, the graph is a tree with loops attached, so $\Pic_p^0(Y_0)=0$.
It follows that, in fact for every $n\ge0$,
\begin{equation}\label{eq:example-one-growth}
\operatorname{ord}_p\bigl(|\Pic^0(Y_n)|\bigr)=p^n+n-1.
\end{equation}

\subsection{Finite and infinite inertia}\label{subsec:example-mixed-inertia}

Let $X$ now be the path with vertices $v_2,v_1,v_3$ and edges joining $v_1$
to each of $v_2,v_3$. Give both oriented edges voltage $1$, and set
\[
\wtI_{v_1}=1,\qquad
\wtI_{v_2}=G,\qquad
\wtI_{v_3}=\Gamma.
\]
Thus $A=\{v_1\}$, $B=\{v_2\}$, and $C=\{v_3\}$.
Although all voltages are trivial, every finite graph $Y_n$ is connected.
Indeed, its vertices can be written as
\[
x_{h,i},\qquad y_i,\qquad z_h
\quad(h\in\mathbb Z/p\mathbb Z,\ i\in\mathbb Z/q\mathbb Z),
\]
where $q=p^n$ and its edges join $x_{h,i}$ to $y_i$ and to $z_h$. Therefore $Y_n$ is
obtained from the complete bipartite graph $K_{p,q}$ by subdividing each
edge once.

For the ordering $(v_1,v_2,v_3)$ one has
\[
A_{\wtalpha}=\begin{pmatrix}0&1&1\\1&0&0\\1&0&0\end{pmatrix},
\qquad D=\operatorname{diag}(2,1,1).
\]
The presentation at the top of the tower is
\begin{equation}\label{eq:example-two-presentation}
R^4\xrightarrow{\displaystyle
\begin{pmatrix}
2&-N&0&0\\
-1&N&\sigma-1&0\\
-1&0&0&T
\end{pmatrix}}
R^3\longrightarrow M\longrightarrow0.
\end{equation}
Its maximal minors are $0$, $NT$, $2(\sigma-1) T$, and $0$. Since $p$ is odd,
\begin{equation}\label{eq:example-two-fitting}
\Fitt_R(M)=T(N,\sigma-1)
         =T(p,\sigma-1).
\end{equation}
This also follows directly from Theorem~\ref{thm:EMC}: deleting the row and
column of $v_3$ gives
\[
\ThetaStar_{Y_\infty/X}(1)
=\det\begin{pmatrix}2&-1\\-N&N\end{pmatrix}=N,
\qquad
\ThetaStar_{Y_\infty^{\{v_2\}}/X}(1)=2,
\]
whereas the two inertia ideals are $(\sigma-1)$ and $(T)$.
As $T$ is a non-zero-divisor in $R$ and $(p,\sigma-1)$ requires two
generators, the ideal in \eqref{eq:example-two-fitting} is nonprincipal.

We next describe $M$ as an ordinary Iwasawa module. Write $x,y,z$ for the
images of the basis vectors in \eqref{eq:example-two-presentation}.
Eliminating $x$ by $2x=y+z$ gives
\begin{equation}\label{eq:example-two-relations}
(\sigma-1) y=0,\qquad Tz=0,\qquad py=Nz.
\end{equation}
Put $z_h=\sigma^h z$ for $0\le h<p$. As a $\Lambda$-module, $M$ consequently
has generators $y,z_0,\ldots,z_{p-1}$ and relations
\[
Tz_h=0\quad(0\le h<p),\qquad
py=z_0+\cdots+z_{p-1}.
\]
Eliminating $z_{p-1}$ shows that
\[
M\simeq\Lambda/(pT)\oplus
       \bigl(\Lambda/(T)\bigr)^{p-1}.
\]
Each $z_h$ has degree $1$. Replacing $y$ by $y-z_0$ and $z_h$ by $z_h-z_0$
for $1\le h\le p-2$ therefore gives
\begin{equation}\label{eq:example-two-degree-zero}
M^0\simeq\Lambda/(pT)\oplus
       \bigl(\Lambda/(T)\bigr)^{p-2},\qquad
\charideal_\Lambda(M^0)=(pT^{p-1}).
\end{equation}
In particular, $\mu(M^0)=1$ and $\lambda(M^0)=p-1$.

The finite-level Fitting ideal can be computed just as explicitly. Put
\[
\nu_n=1+\gamma+\cdots+\gamma^{q-1}\in R_n.
\]
At level $n$, the additional principal-divisor relation at $v_3$ must be
retained. After eliminating $x$ it is $qz=\nu_n y$. Thus a presentation of
$\Pic_p(Y_n)$ over $R_n$ is
\[
R_n^4\xrightarrow{\displaystyle
\begin{pmatrix}\sigma - 1&0&p&-\nu_n\\0&T&-N&q\end{pmatrix}}
R_n^2\longrightarrow\Pic_p(Y_n)\longrightarrow0.
\]
Its maximal minors give, in agreement with Theorem~\ref{thm:finite-Fitt},
\begin{equation}\label{eq:example-two-finite-fitting}
\Fitt_{R_n}\bigl(\Pic_p(Y_n)\bigr)
=((\sigma - 1) T,pT,q(\sigma - 1)).
\end{equation}
Indeed, $N(\sigma - 1)=T\nu_n=0$, and the only remaining minor is $pq-N\nu_n$.
Writing $N-p=(\sigma - 1) u$ and $\nu_n-q=Tv$ shows that
\[
N\nu_n-pq=q(\sigma - 1) u+pTv+(\sigma - 1) Tuv,
\]
so this last minor also belongs to the displayed ideal.

There is also a direct calculation of the finite Picard groups.
The principal-divisor relations at the middle vertices are
\[
2x_{h,i}=y_i+z_h.
\]
Since $p$ is odd, we may eliminate these generators by substituting
$x_{h,i}=(y_i+z_h)/2$ into the relations at the remaining vertices.
After multiplying by $2$, these relations become
\[
py_i-\sum_{h=0}^{p-1}z_h=0,
\qquad
qz_h-\sum_{i=0}^{q-1}y_i=0.
\]
These are precisely the principal-divisor relations for $K_{p,q}$.
Consequently,
\[
\Pic_p(Y_n)\simeq\Pic_p(K_{p,q})
\quad\text{as $\mathbb Z_p$-modules}.
\]
For $n\ge1$, elementary row and column operations on a reduced Laplacian of
$K_{p,q}$ give the Smith diagonal
\[
1,\ 1,\
\underbrace{p,\ldots,p}_{q-2},\
\underbrace{q,\ldots,q}_{p-2},\ pq.
\]
Consequently,
\begin{equation}\label{eq:example-two-finite-jac}
\Pic_p^0(Y_n)\simeq
(\mathbb Z/p\mathbb Z)^{q-2}
\oplus(\mathbb Z/p^n\mathbb Z)^{p-2}
\oplus\mathbb Z/p^{n+1}\mathbb Z
\qquad(n\ge1).
\end{equation}
When $n=0$, the graph is a tree and its degree-zero Picard group is trivial.
For every $n\ge0$ we therefore have
\begin{equation}\label{eq:example-two-growth}
\operatorname{ord}_p\bigl(|\Pic^0(Y_n)|\bigr)
=p^n+(p-1)n-1.
\end{equation}

\end{document}